\documentclass[11pt]{article}
\usepackage{amssymb}
\usepackage{amsmath}
\input{epsf.sty}  
\usepackage{color}

\newtheorem{remark}{Remark}[section] 
\newtheorem{application}{Application}

\newtheorem{proposition}{Proposition}[section]
\newtheorem{corollary}{Corollary}[section]

\newcommand{\E}{\mathbb{E}}
\newcommand{\Prob}{\mathbb{P}}
\newenvironment{proof}{\begin{trivlist}
\item[\hspace{\labelsep}{\bf\noindent Proof. }]}
{\end{trivlist}}
\title{\huge\bf Compound Poisson process \\ with a Poisson subordinator}
\author{
{\normalsize
\bf Antonio Di Crescenzo$^{(1)}$, \ Barbara Martinucci$^{(1)}$, \ Shelemyahu Zacks$^{(2)}$} \\
\normalsize
(1) Dipartimento di Matematica, 
Universit\`a degli Studi di Salerno \\
\normalsize
84084 Fisciano (SA),
Italy \\
\normalsize
email: \{adicrescenzo,\,bmartinucci\}@unisa.it \\
\normalsize
(2) Department of Mathematical Sciences, 
Binghamton University  \\
\normalsize
Binghamton, NY 13902-6000,
USA \\
\normalsize
email: shelly@math.binghamton.edu
}
\date{\bf First published in  {\em Journal of Applied Probability}, Vol.\ 52, No.\ 2, p.\ 360-374  \\
\copyright\ 2015 by the 
 Applied Probability Trust}
\begin{document}
\maketitle

\begin{abstract}
A compound Poisson process whose randomized time is an independent Poisson process is called 
compound Poisson process with Poisson subordinator. We provide its probability distribution, which is 
expressed in terms of the Bell polynomials, and investigate in detail both the special cases in which 
the compound Poisson process has exponential jumps and normal jumps. Then for the iterated 
Poisson process we discuss some properties and provide convergence results to a Poisson process. 
The first-crossing-time problem for the iterated Poisson process is finally tackled in the cases of 
(i) a decreasing and constant boundary, where we provide some closed-form results, and (ii) a linearly 
increasing boundary, where we propose an iterative procedure to compute the first-crossing-time 
density and survival functions. 



\end{abstract}
\section{Introduction}
Stochastic processes evaluated at random times are receiving increasing attention in various applied 
fields. There are many examples of processes with random times that could be modeled in this manner:
\begin{description}
\item{\em (i)} in reliability theory, the life span of items subject to accelerated conditions, 
\item{\em (ii)} in econometrics, the composition of the price of a security and the effective economic time, 
\item{\em (iii)} in queueing theory, the number of customers joining the system during specific service periods, 
\item{\em (iv)} in statistics, for the random sampling of stochastic processes. 
\end{description}
One of the first papers in this field is  Lee and Whitmore \cite{LeeWhi93}, who studied general properties 
of processes directed by randomized times. In the literature special attention is given to the case of a Poisson 
process with randomized time. Another earlier example of a stochastic process with a Poisson subordinator is 
given in Pickands \cite{Pi71}. Cox processes are also Poisson processes with randomized times. 
These are also called doubly stochastic Poisson processes 
(see  \cite{Se72} and  \cite{Gr76}). Kumar  {\em et al.}\ \cite {KNV2011} considered 
time-changed Poisson processes where the subordinator is an inverse Gaussian process, i.e.\ the 
hitting time process of the standard Brownian motion. There are fractional Poisson processes, which 
are processes with randomized continuous fractional diffusions (see \cite{BeOr2009}, 
\cite{BeOr2010}, and \cite{MaGoSc04}). There are examples where the random 
times are Brownian motions (see \cite{BeOr2012}). For other types of iterated Poisson 
process; see, e.g.\ \cite{HoSt99}. 
\par
Let us consider a compound Poisson process (CPP), 
$$
 Y(t)=\sum_{n=1}^{M(t)} X_n, \qquad t>0
$$
with $M(t)$ a Poisson process, and $X_n$ a sequence of independent and identically distributed 
(i.i.d.)\ random variables independent of $M(t)$. 
In this present paper we investigate the distribution of the process $Z(t)=Y[N(t)]$, where $N(t)$
is an independent Poisson process. Such a process is called a CPP with a Poisson subordinator. 
We provide explicit equations of this distribution, its moments and other characteristics. Precisely, 
after some preliminary results concerning the distribution of the CPP $Y(t)$ given in Section 2, 
in Section 3 we obtain the probability distribution of the process $Z(t)$ in terms of the Bell polynomials. 
Then in Sections 4 and 5, we develop the equations of the distribution of $Z(t)$ in the special cases 
when $Y(t)$ has exponential jumps and normal jumps, respectively.
\par
In order to demonstrate the usefulness of the CPP with a Poisson subordinator we now provide two examples 
of application of in biomedical research and ecology. 
\begin{application}\label{app:bio}{
Assume that patients having the same disease arrive at a clinic at random times, according to the Poisson 
process $N(t)$. Each patient receives the same treatment and is under inspection for one time unit (a year). 
The symptoms of the disease may reoccur at random times during the inspection period, 
according to the independent Poisson process $M(t)$. The total number of occurrences and their severity 
is distributed for each patient as the random variable (RV) $Y(1)$. Thus, the process $Z[N(t)]$  
represents the total damage inflicted by the disease to patients after their inspection period.
}\end{application}
\begin{application}\label{app:eco}{ 
Let $N(t)$ be a Poisson process describing the number of animals caught in $[0,t]$ during an investigation 
in ecology. A radio transmitter is attached to each animal for a period of one time unit. A transmitter sends 
information at random, according to an independent Poisson process $M(t)$. The total number of 
occurrences and the amount of information is distributed for each animal as the RV $Y(1)$. Thus, the 
process $Z[N(t)]$  represents the total amount of information transmitted by the animals after the 
investigation period. 
}\end{application}
\par
The second part of the paper is devoted to a special case, namely the iterated Poisson process, which  
is a Poisson process whose randomized time is another independent Poisson process. This is actually 
a CPP whose independent jumps are discrete RVs having a Poisson 
distribution. An example in queueing theory involving the iterated Poisson process is 
provided in Application \ref{app:1}. We remark that 
the probability law, the governing equations, its representation as a random sum, and 
various generalizations of the iterated Poisson process have been studied in   
\cite{OrPo2010} and \cite{OrPo2012}. For such a process, in Section 6 we express in series form the 
mean sojourn time in a fixed state. Moreover, we also find conditions under which it converges to a 
regular Poisson process. Stopping time problems for the iterated Poisson process are finally 
studied in Section 7. In the case of constant boundaries we obtain the first-crossing-time density 
and the mean first-crossing time in closed form. For linear increasing boundaries we develop a 
computationally effective procedure able to determine iteratively the first-hitting-time density and, 
in turn, the corresponding survival function. 
%
\section{Preliminaries}
In the present section we bring some well-known results about the distribution of the CPP, $Y(t)$, 
and Bell polynomials, which are used later in the distribution of $Z(t)=Y[N(t)]$. 
\par
Consider the compound Poisson process $\{Y(t),t\geq 0\}$ defined as 
\begin{equation}
 Y(t)=\sum_{n=1}^{M(t)} X_n,
 \qquad t>0,
 \label{eq:defYt}
\end{equation}
where $\{M(t),t\geq 0\}$ is a Poisson process with intensity $\mu$, and $\{X_n, n\geq 1\}$ 
is a sequence of i.i.d.\ RVs independent of $\{M(t),t\geq 0\}$. 
The probability mass  function of $M(t)$ at $m$ will be denoted as 
\begin{equation}
 p(m;\mu t)=\Prob\{M(t)=m\}={\rm e}^{-\mu t} \frac{(\mu t)^m}{m!}, 
 \qquad m=0,1,\ldots.
 \label{eq:defpMt}
\end{equation}
We assume that $Y(0)=0$ and $M(0)=0$. Let $H_Y(y;t)$ denote the cumulative distribution 
function (CDF) of $Y(t)$, $y\in {\mathbb R}$ and $t>0$. 
From (\ref{eq:defYt}) it follows that if $M(t)=0$ then $Y(t)=0$ so that the distribution of  $Y(t)$ 
has an atom at $0$ with $\Prob\{Y(t)=0\}\geq p(0;\mu t)={\rm e}^{-\mu t}$, $t\geq 0$, 
where $\Prob$ denotes the probability measure. Moreover, recalling (\ref{eq:defpMt}) 
it follows that the CDF of $Y(t)$ is given by the following Poisson mixture:
\begin{equation}
\begin{split}
 H_Y(y;t)&=
 \sum_{m=0}^{+\infty} 
 p(m;\mu t)\,F^{(m)}_X(y), \\
 &={\bf 1}_{\{y\geq 0\}} {\rm e}^{-\mu t}
 +\sum_{m=1}^{+\infty} 
 p(m;\mu t)\,F^{(m)}_X(y),
 \qquad y\in{\mathbb R}, \;\; t\geq 0. 
 \end{split}
 \label{eq:esprHyt}
\end{equation}
Note that in  (\ref{eq:esprHyt}) and  throughout the paper, $g^{(m)}$ denotes the $m$-fold convolution 
of a given function $g$. The indicator function is denoted by ${\bf 1}_{\{\cdot \}}$. 
\par
If the RVs $X_i$ are absolutely continuous with probability density function $f_X(\cdot)$ then 
due to (\ref{eq:esprHyt}), 
\begin{description}
\item{(i)} the absolutely continuous component of the probability law of $Y(t)$ is expressed by the density
$$
 h_Y(y;t)=\sum_{m=1}^{+\infty} p(m;\mu t)\,f^{(m)}_X(y),
 \qquad y\neq 0, \;\; t> 0,
$$ 
\item{(ii)}  the discrete component is given by $\Prob\{Y(t)=0\}=p(0;\mu t)={\rm e}^{-\mu t}$, $t\geq 0$. 
\end{description}
\par
Let $M_X(s):=\E\left\{{\rm e}^{s X}\right\}$ be the moment generating function of $X_i$, where $s$ 
is within the region of convergence of  $M_X(s)$ and $\E$ is the expectation. 
From (\ref{eq:defYt}) it follows that the moment generating function of $Y(t)$ is   
\begin{equation}
 \E\left\{{\rm e}^{s Y(t)}\right\}
 =\exp\left\{-\mu t [1-M_X(s)]\right\},
 \qquad  t\geq 0.
 \label{eq:mgfYt}
\end{equation}
\par
In the following section we study the process $Y(t)$ when the time is randomized 
according to a Poisson process. Precisely, we consider the stochastic process $\{Y[N(t)], t\geq 0\}$, 
where  $\{N(t),t\geq 0\}$ is a Poisson process with intensity $\lambda$, independent of $Y(t)$. 
We will express the c.d.f.\ of $Z(t)$ in terms of Bell polynomials
\begin{equation}
 B_n(x)=\sum_{k=0}^{\infty} k^n \, \frac{x^k}{k!}\,{\rm e}^{-x}, \qquad n\geq 0.
 \label{eq:Bnxserie}
\end{equation}
We therefore recall some properties of these polynomials (see, e.g.\  \cite{Com74}). 
Let $B_n(x)$, $x\geq 0$, represent the $n$th moment of a Poisson distribution with mean $x$. 
Obviously, $B_0(x)=1$, $B_1(x)=x$, and $B_2(x)=x+x^2$. Generally, $B_n(x)$ is a polynomial 
of degree $n$ of the form
$$
 B_n(x)=x+a_2 x^2+\ldots+a_{n-1} x^{n-1}+x^n,\qquad n\geq 2.
$$
We can determine $B_n(x)$ by differentiating the moment generating function 
$\exp\big\{-x(1-{\rm e}^{\theta})\big\}$ with respect to $\theta$, $n$ times, and evaluating the 
derivative at $\theta=0$. From (\ref{eq:Bnxserie}) we immediately obtain the derivative
\begin{equation}
B'_n(x)=-B_n(x)+\frac{B_{n+1}(x)}{x},\qquad n\geq 0.
\label{bellderivata}
\end{equation}
Hence, due to (\ref{bellderivata})  the following recursive formula holds: 
\begin{equation}
B_{n+1}(x)=x[B'_n(x)+B_{n}(x)],\qquad n\geq 0,
\label{bellrecursive}
\end{equation}
which is useful in obtaining the explicit form of the Bell polynomials. For instance, 
according to (\ref{bellrecursive}), we have  
\begin{eqnarray*}
&& B_3(x)=x+3 x^2+x^3,\\
&& B_4(x)=x+7 x^2+6x^3 +x^4,\\
&& B_5(x)=x+15 x^2+25 x^3+10 x^4+x^5,
\end{eqnarray*}
etc.  
Moreover, we can also express these polynomials by the Dobi\'nski formula  
\begin{equation}
 B_n(x)=\sum_{k=0}^{n} S_2(n,k)\,x^k 
 \qquad \hbox{for }n=0,1,\ldots,
 \label{eq:dobinski}
\end{equation}
where 
$$
 S_2(n,k)=\left\{ 
 {n \atop k}
 \right\}
 ={1\over k!}\sum_{i=0}^{k}(-1)^i {k\choose i} (k-i)^n 
 \qquad \hbox{for }k=0,1,\ldots,n,
$$
are the Stirling numbers of the second kind. 
\section{The distribution of $Z(t)=Y[N(t)]$}
First of all we point out that, due to (\ref{eq:defYt}), the process $\{Y[N(t)],t\geq 0\}$ is identically  
distributed to the CPP 
\begin{equation}
 Z(t):=\sum_{n=0}^{N(t)} W_n,
 \qquad t\geq 0,
 \label{eq:defZt}
\end{equation}
where $W_0= 0$ almost surely (a.s.), and $W_1,W_2,\ldots$ are i.i.d.\ RVs distributed as 
$$
 Y(1)=\sum_{n=1}^{M(1)}X_n.
$$ 
\par
Note that the CDF of $W_n$, $n\geq 1$, is $H_Y(y;1)$. The $n$-fold convolution of this 
CDF is $H_Y(y;n)$, $n\geq 1$. We denote by $H_Z(z;t)$ the CDF of $Z(t)$. 
\begin{proposition}\label{proposition:HZ}
For all $z\in\mathbb{R}$ and $t>0$ the CDF of process (\ref{eq:defZt}) is 
\begin{equation}
 H_Z(z;t)=\exp\left\{-\lambda t(1-{\rm e}^{-\mu})\right\}
 \left[ {\bf 1}_{\{z\geq 0\}}
 +\sum_{n=1}^{+\infty} 
 {\mu^n\over n!}\,B_n\left(\lambda t {\rm e}^{-\mu}\right) F^{(n)}_X(z)\right].
 \label{eq:esprHzt}
\end{equation}
\end{proposition}
\begin{proof}
From  (\ref{eq:defZt}) and  (\ref{eq:esprHyt})  we have 
\begin{equation}
\begin{split}
 H_Z(z;t) &= \sum_{n=0}^{+\infty} p(n;\lambda t)\,H_Y(z;n) \\
 &=\sum_{n=0}^{+\infty} p(n;\lambda t) 
 \left[
 {\bf 1}_{\{z\geq 0\}}{\rm e}^{-\mu n}
 +\sum_{m=1}^{+\infty} p(m;\mu n)\,F^{(m)}_X(z)
 \right],
 \qquad z\in\mathbb{R}, \;\; t>0.
\end{split}
\nonumber
\end{equation}
In addition, 
$$
 \sum_{n=0}^{\infty} p(n; \lambda t) {\rm e}^{-\mu n}={\rm e}^{-\lambda t (1-{\rm e}^{-\mu})};
$$
also, since all terms are non-negative, and recalling (\ref{eq:Bnxserie}), 
\begin{align*}
 \sum_{n=0}^{\infty} p(n; \lambda t)\sum_{m=1}^{\infty} p(m; \mu n) F_X^{(m)}(z)
 & =\sum_{m=1}^{\infty} \frac{\mu^m}{m!} F_X^{(m)}(z)\sum_{n=0}^{\infty} n^m {\rm e}^{-\lambda t} 
\frac{(\lambda t {\rm e}^{-\mu})^n}{n!}
\\
& ={\rm e}^{-\lambda t (1-{\rm e}^{-\mu})} \sum_{m=1}^{\infty} \frac{\mu^m}{m!} B_m(\lambda t {\rm e}^{-\mu})F_X^{(m)}(z).
\end{align*}
This proves (\ref{eq:esprHzt}). 
\end{proof}
\begin{corollary}
If $F_X$ is absolutely continuous with density $f_X$ then the density of (\ref{eq:esprHzt}) for any 
$z\neq 0$ and $t>0$ is 
\begin{equation}
 h_Z(z;t)=\exp\left\{-\lambda t(1-{\rm e}^{-\mu})\right\}
 \sum_{n=1}^{+\infty} {\mu^n\over n!}\,B_n\left(\lambda t {\rm e}^{-\mu}\right)f^{(n)}_X(z). 
 \label{eq:denshZt}
\end{equation}
\end{corollary}
\begin{remark}\rm 
Note that if $F_X$ is continuous at $0$ then $H_Z(z;t)$ has a jump of size 
${\rm e}^{-\lambda t(1-{\rm e}^{-\mu})}$ at $z=0$. 
\end{remark}
\begin{remark}\rm 
We can easily prove that $H_Z(z;t)$ given by (\ref{eq:esprHzt}) is indeed a CDF.
\end{remark}
\par
Recalling that $W_i$ is distributed as $Y(1)$, by setting $t=1$ in (\ref{eq:mgfYt})  
we obtain the following moment generating function: 
$$
 M_W(s):=\E\left\{{\rm e}^{s W_i}\right\}=\exp\{-\mu[1-M_X(s)]\}, \qquad i=1,2,\ldots.
$$
Hence, due to (\ref{eq:denshZt}) the Laplace-Stieltjes transform of $Z(t)$ is 
$$
 \E\left\{ {\rm e}^{-\theta Z(t)}\right\}
 =\exp\left\{-\lambda t[1-M_W(-\theta)]\right\}
 =\exp\left\{- t \, \Psi(\theta) \right\}, 
 \qquad t>0,
$$
where 
$$
 \Psi(\theta)=\lambda \left[1-{\rm e}^{-\mu [1-M_X(-\theta)]}\right].
$$
We recall that $\{Z(t),t\geq 0\}$ is a CPP time-changed by a Poisson process, 
and thus it is a L\'evy process (see, e.g.\ [1, Theorem 1.3.25]). 
\par
Let us now focus on the mean and the variance of $Z(t)$. From  (\ref{eq:defZt}) we have 
$$
 \E\{Z(t)\}=\lambda \mu \xi t, \qquad 
 {\rm var}[Z(t)]=\lambda \mu [ \sigma^2 +(\mu+1)\xi^2] t,
 \qquad t\geq 0,
$$
where we have set 
$$
 \xi=\E\{X_1\}\quad \hbox{and} \quad \sigma^2={\mathbb V}ar(X_1).
$$ 
Finally, by the strong law of large numbers,  the following asymptotic result holds for process $Z(t)$:
$$
 \lim_{t\to +\infty}{Z(t)\over t}=\lambda \mu \xi\qquad {\rm a.s.}
$$ 
\par
In the following sections we consider three special cases of interest, in which the distributions of the 
jump variables $X_i$ are assumed to be deterministic, exponential, and normal.  
%
\section{The compound Poisson process with exponential jumps}
In this section we consider the CPP (\ref{eq:defYt}) when $\{X_n, n\geq 1\}$ are i.i.d.\ 
exponentially distributed RVs with parameter $\zeta$, i.e.\ 
$f_X(x)=\zeta {\rm e}^{-\zeta x}\,{\bf 1}_{\{x\geq 0\}}$. In this case, we have for $x\geq 0$, 
$$
 F_X(x)=1-p(0;\zeta x), \qquad F_X^{(m)}(x)=1-P(m-1;\zeta x),
 \qquad m=1,2\ldots,
$$
where $p(n;\lambda)$ is defined in (\ref{eq:defpMt}) and $P(n;\lambda)$ is the CDF given by 
$$
 P(n;\lambda)=\sum_{i=0}^{n}p(i;\lambda). 
$$ 
Hence, due to (\ref{eq:esprHyt}) the CDF of $Y(t)$, $t\geq 0$, is  
$$
 H_Y(y;t)= {\rm e}^{-\mu t}+ \sum_{m=1}^{+\infty} p(m;\mu t)\,[1-P(m-1;\zeta y)],
 \qquad y\geq 0,
$$
Thus, the atom of such a function is $H_Y(0;t)={\rm e}^{-\mu t}$. 
Generally, for $n=1,2,\ldots$, 
$$
 H_Y^{(n)}(y;1)=H_Y(y;n)
 =1- \sum_{m=1}^{+\infty} p(m;n\mu)\,P(m-1;\zeta y),
 \qquad y\geq 0.
$$
According to Proposition \ref{proposition:HZ}, the CDF of $Z(t)$ can be 
expressed as   
\begin{eqnarray}
 && \hspace*{-1.6cm}
 H_Z(z;t)= \sum_{n=0}^{+\infty} p(n;\lambda t)
 \left[1- \sum_{m=1}^{+\infty} p(m;n\mu)P(m-1;\zeta z)\right]
 \nonumber  \\
 && \hspace*{-0.1cm}
 = 1-\sum_{n=0}^{+\infty} p(n;\lambda t) \sum_{m=1}^{+\infty} p(m;n\mu)P(m-1;\zeta z)
 \nonumber  \\
 && \hspace*{-0.1cm}
 =1-\sum_{m=1}^{+\infty} p_m(t)P(m-1;\zeta z),
 \label{eq:HZtesp}
\end{eqnarray}
since $p_m(t)=\sum_{n=0}^{+\infty} p(n;\lambda t) p(m;n\mu)$ due to (\ref{equation:pnt}). 
Additional manipulations yield the alternative equation 
$$
  H_Z(z;t) = \sum_{j=0}^{+\infty} p(j;\zeta z)  \sum_{m=0}^{j} p_m(t).
$$
\par
We pinpoint that the atom at 0 of (\ref{eq:HZtesp}) is $H_Z(0;t) = p_0(t)$. 
Since $({\rm d}/{\rm d}z)P(m-1;\zeta z)=-\zeta  p(m-1; \zeta z)$, 
the density of $H_Z(z;t)$ for $z>0$ and $t>0$, is expressed as 
\begin{equation}
 h_Z(z;t)=\zeta \sum_{m=1}^{+\infty}   p_m(t) p(m-1;\zeta z). 
 \label{eq:denshZtesp}
\end{equation}
Some plots of such density are presented in Figure \ref{eq:denshZesp}, where the probability 
mass $\int_0^{+\infty}h_Z(z;t)\,{\rm d}z$ is shown in the caption for some choices of $t$.
%
\begin{figure}[t]
\ (a) \hspace{7cm} \ (b)\\
\centerline{
\epsfxsize=7.5cm
\epsfbox{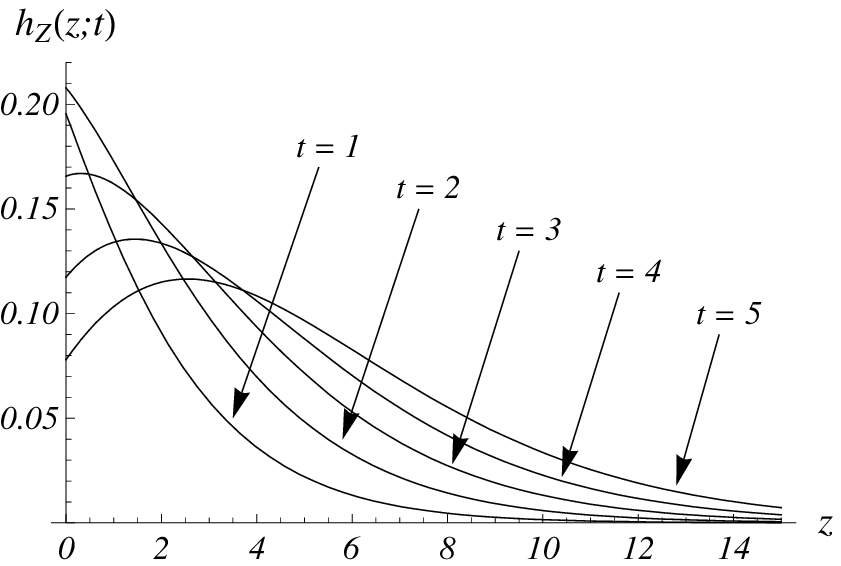}
$\;$
\epsfxsize=7.5cm
\epsfbox{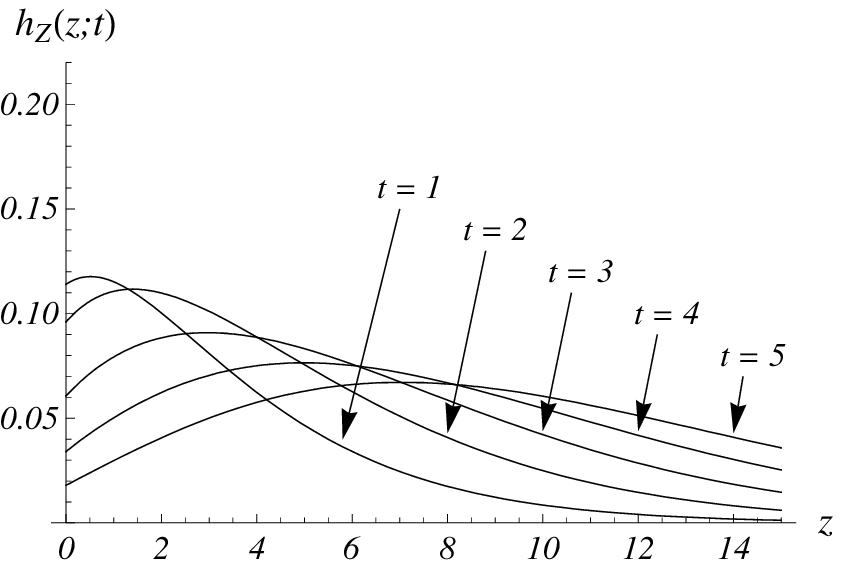}
}
\caption{\small Density (\ref{eq:denshZtesp}) for  $\mu=1$, $\zeta=1$ and various choices of $t$, 
with (a) $\lambda=1$, and (b) $\lambda=2$. 
The corresponding probability mass of $h_Z(z;t)$ is 0.4685,  0.7175, 0.8499, 0.9202, 0.9576 
in (a) and 0.7175, 0.9202, 0.9775, 0.9936, 0.9982 in (b) for increasing values of $t$.
}
\label{eq:denshZesp}
\end{figure}
\par
Finally, we note that the expected value of $Z(t)$ is $\E\{Z(t)\}=\lambda\mu t/\zeta$.   
\section{The compound Poisson process with normal jumps}
Let us now investigate the CPP (\ref{eq:defYt}) when $\{X_n, n\geq 1\}$ are 
i.i.d.\ normally distributed RVs with mean $\eta\in\mathbb{R}$ and variance $\sigma^2>0$, so     
$$
 F_X(x)=\Phi\Big(\frac{x-\eta}{\sigma}\Big), 
 \qquad 
 F_X^{(n)}(x)=\Phi \Big(\frac{x-n\eta}{\sigma\sqrt{n}}\Big),
 \qquad x\in{\mathbb R},
$$
where $\Phi(\cdot)$ is the standard normal distribution function. 
Hence, for all $y\in {\mathbb R}$, 
$$
 H_Y^{(n)}(y;1)= H_Y(y;n)
 ={\rm e}^{-n\mu} {\bf 1}_{\{y\geq 0\}}
 + \sum_{m=1}^{+\infty} p(m;n\mu)\Phi \Big(\frac{y-m\eta}{\sigma\sqrt{m}}\Big),
 \qquad  t\geq 0.
$$
Note that the above distribution has support on $(-\infty,+\infty)$, and it is absolutely continuous on 
all $y\neq 0$, with a jump (atom) at $y=0$. According to (\ref{eq:esprHzt}), the CDF of $Z(t)$ is 
\begin{equation}
 H_Z(z;t) = p_0(t){\bf 1}_{\{z\geq 0\}}
 +\sum_{n=1}^{+\infty} p_n(t) \Phi \Big(\frac{z-n\eta}{\sigma\sqrt{n}}\Big),
 \qquad z\in\mathbb{R}, \;\; t\geq 0,
\label{eq:hznormal}
\end{equation}
where $p_n(t)$  is given in (\ref{equation:pnt}). In Figure \ref{fig:3} we show some plots of $H_Z(z;t)$. 
In this case the L\'evy exponent of $\{Z(t), t\geq 0\}$ is 
$$
 \Psi(\theta)=\lambda \left[1-\exp\left\{-\mu\left(1-{\rm e}^{-\eta \theta+\sigma^2\theta^2/2}\right)\right\}\right].
$$
In conclusion, from (\ref{eq:hznormal}) it follows that the density of $Z(t)$, at $z\neq 0$ is 
$$
 h_Z(z;t)=\frac{1}{\sigma} \sum_{n=1}^{+\infty} p_n(t)\frac{1}{\sqrt{n}}
 \phi \Big(\frac{z-n\eta}{\sigma\sqrt{n}}\Big),
$$
where $\phi(\cdot)$ is the standard normal density. 
\begin{figure}[t]
\ (a) \hspace{7cm} \ (b)\\
\centerline{
\epsfxsize=7.5cm
\epsfbox{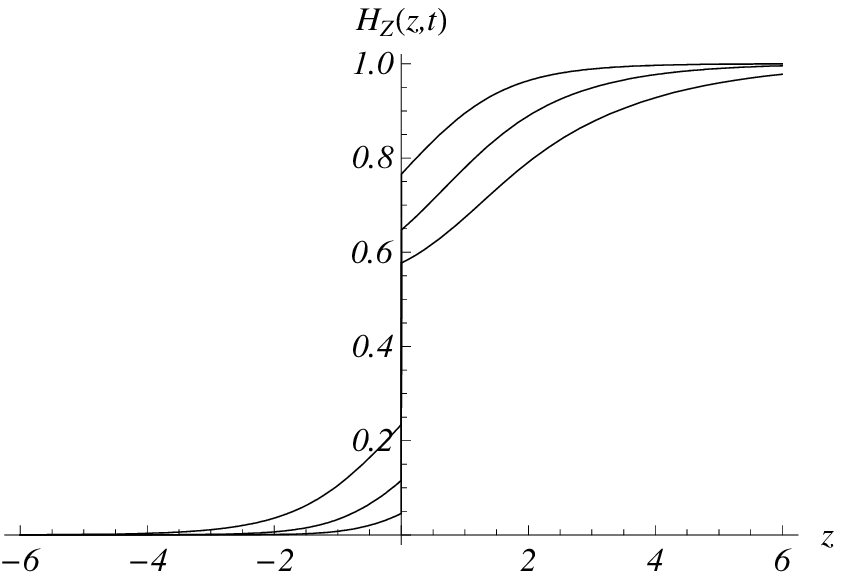}
$\;$
\epsfxsize=7.5cm
\epsfbox{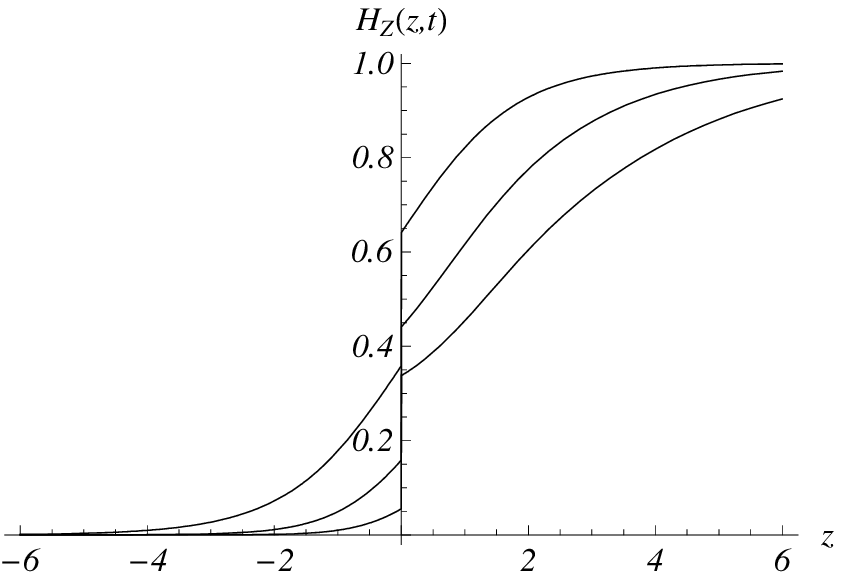}
}
\caption{\small Distribution function (\ref{eq:hznormal}) for $\lambda=1$, $\mu=1$,  $\sigma=1$ and 
$\eta=0, 0.5, 1$ (from top to bottom), with (a) $t=1$, and (b) $t=2$.
}
\label{fig:3}
\end{figure}
\section{The iterated Poisson process}\label{section:iPp}
In this section we suppose that in Eq.\ (\ref{eq:defZt}) the RVs 
$W_n$, $n=1,2,\ldots$, have a Poisson distribution with parameter $\mu$. 
This special case corresponds to the assumption that $X_i=1$, $i\geq 1$ a.s.\ so that  
$$
F_X(x)={\bf 1}_{\{x\geq 1\}}, 
 \qquad 
 F^{(n)}_X(x)={\bf 1}_{\{x\geq n\}}, \quad n=1,2,\ldots,
$$
and, thus, $Y(t)=M(t)$, $t\geq 0$, a.s. 
Hence, in this case $\{Z(t),t\geq 0\}$ is a Markovian jump process over the set of nonnegative 
integers, with right-continuous nondecreasing sample-paths. This is called an `iterated Poisson 
process' with parameters $(\mu,\lambda)$, since it can be expressed as a Poisson process 
$M(t)$ with intensity $\mu$ whose randomized time is an independent Poisson process $N(t)$ 
with intensity $\lambda$, i.e.\ $Z(t)=M[N(t)]$. We point out that various results on this process 
have been obtained in \cite{OrPo2010} and \cite{OrPo2012}. 
\par
In analogy with (\ref{eq:denshZt}), the probability function of the iterated Poisson process 
for $t>0$ is (cf.\ [12, Equation  (8)]) 
\begin{equation}
 p_n(t)=\Prob\{Z(t)=n\}
 =\exp\left\{-\lambda t \left(1-{\rm e}^{-\mu}\right)\right\}\,{\mu^{n}\over n!}\,
 B_n\left(\lambda t {\rm e}^{-\mu}\right), 
 \qquad n=0,1,\ldots.
 \label{equation:pnt}
\end{equation}
Note that $ \sum_{n=0}^{+\infty}p_n(t)=1$ for all $t>0$. We recall that $B_0(x)=1$, 
and that $B_n(x)$ can be expressed in closed form by means of the Dobi\'nski formula 
(\ref{eq:dobinski}). Some plots of $p_n(t)$ are shown in Figure \ref{fig:probt}, for $t=1,2,3$. 
\begin{figure}[t]
%
\centerline{
\epsfxsize=7.5cm
\epsfbox{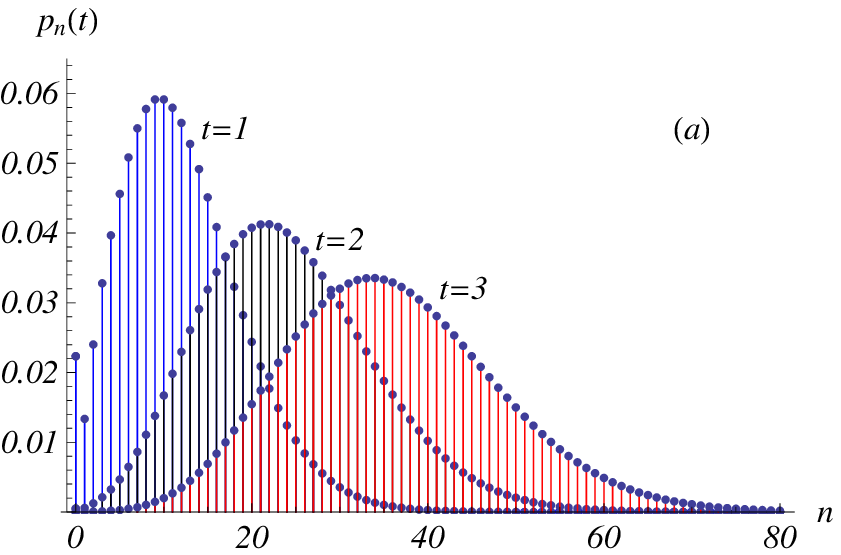}
$\;$
\epsfxsize=7.5cm
\epsfbox{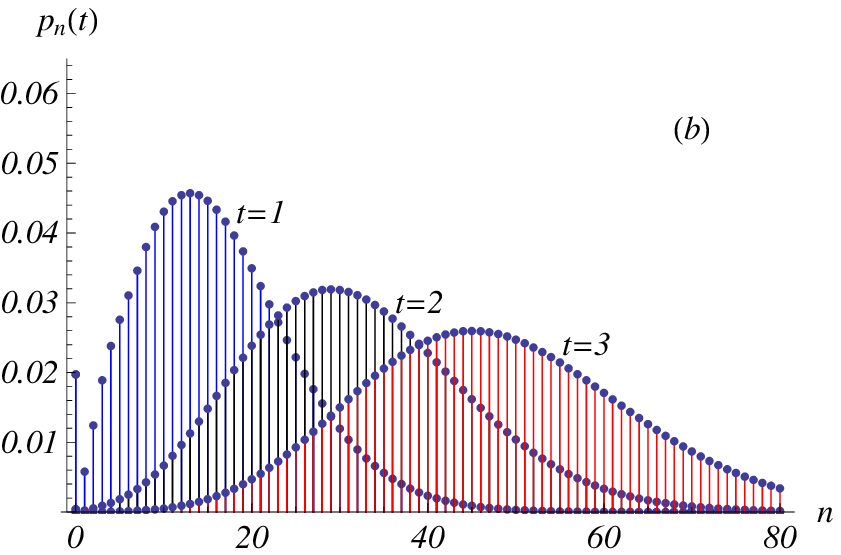}
}
\caption{\small Probability distribution (\ref{equation:pnt}) for $\lambda=4$ and 
for (a) $\mu=3$ and (b) $\mu=4$.}
\label{fig:probt}
\end{figure}
\par
As a special case of (\ref{eq:esprHzt}) we have the discrete CDF of $Z(t)$, $t\geq 0$:  
\begin{equation}
 P_n(t)=H_Z(n;t) 
 =\exp\left\{-\lambda t \left(1-{\rm e}^{-\mu}\right)\right\}
 \sum_{j=0}^n {\mu^{j}\over j!}\,
 B_j\left(\lambda t {\rm e}^{-\mu}\right), 
 \qquad n=0,1,\ldots.
 \label{equation:Pntiniziale}
\end{equation}
Making use of Eqs.\ (\ref{equation:pnt}) and (\ref{eq:dobinski}) the following alternative 
form of $P_n(t)$, $t\geq 0$, can be obtained:   
\begin{equation}
 P_n(t)=\exp\left\{-\lambda\left(1-{\rm e}^{-\mu}\right)t\right\}\,
 \bigg[1+{\bf 1}_{\{n\geq 1\}} \sum_{k=0}^n \left(\lambda{\rm e}^{-\mu}t\right)^k
 \sum_{j=k}^n S_2(j,k)\, {\mu^{j}\over j!}\bigg], \qquad n=0,1,\ldots.
  \label{equation:Pnt}
\end{equation}
\par
The L\'evy exponent of the iterated Poisson process is 
\begin{equation}
 \Psi(\theta)=\lambda\left[1- \exp\{-\mu(1-{\rm e}^{-\theta})\right].
  \label{eq:Psitheta}
\end{equation}
Note that if $\mu\to 0$ and $\lambda\to+\infty$ such that $\lambda\mu\to \xi$ with $0<\xi<+\infty$  
then the right-hand-side of (\ref{eq:Psitheta}) tends to $\xi(1-{\rm e}^{-\theta})$, which is the 
L\'evy exponent of the Poisson process with intensity $\xi$. 
\begin{remark}{\rm 
(i) \ Due to the following well-known recurrence formula of Bell polynomials,
\begin{equation}
 B_n(x)=x\sum_{k=1}^{n} {n-1\choose k-1} B_{k-1}(x), 
 \qquad n=1,2,\ldots,
 \label{eq:recBn}
\end{equation}
the probability distribution of  the iterated Poisson process satisfies the recurrence relation
$$
 p_n(t)={1\over n}\, \lambda {\rm e}^{-\mu} t 
 \sum_{k=1}^n {\mu^{n-k+1}\over (n-k)!}\,p_{k-1}(t),  
 \qquad n=1,2,\ldots, \;\; t>0.
$$
(ii) \  The following conditional probability holds for $0<s<t$ and $k=0,1,\ldots,n$:
\begin{equation}
 \Prob[Z(s)=k\,|\,Z(t)=n]   
 =  {n\choose k}{B_k\left(\lambda{\rm e}^{-\mu}s\right)\,
 B_{n-k}\left(\lambda{\rm e}^{-\mu}(t-s)\right) 
 \over B_n\left(\lambda{\rm e}^{-\mu}t\right)}.
 \label{eq:probcond}
\end{equation}
If $t\to +\infty$ and $s\to +\infty$ with $s/t\to\theta\in(0,1)$, 
then the right-hand-side of (\ref{eq:probcond}) tends to a binomial distributions with 
parameters $n$ and $\theta$, since $B_n(x)\sim x^n$ for $x$ large. 
}\end{remark}
\par
Recalling (see \cite{OrPo2012}) that the mean and the variance of $Z(t)$ are given 
by $\E\{Z(t)\} =  \lambda\mu t$ and ${\rm var}[Z(t)] = \lambda\mu (1+\mu) t$, $t\geq 0$, respectively, we have 
that the iterated Poisson process is overdispersed, since its dispersion index  
$$
 \frac{\mathbb {\rm var}[Z(t)]}{\E\{Z(t)\}}=1+\mu
$$  
is larger than 1 and, in particular, is independent of time $t$. 
\par
Finally, let us denote by $S_n$ the sojourn time of the iterated Poisson process $Z(t)$ in state $n$. 
Since $Z(0)=0$ a.s.\ and  
$$
 \lim_{h\to 0^+}{1\over h}\,\Prob[Z(t+h)\neq k\,|\,Z(t)=k] 
 = \lambda\left(1-{\rm e}^{-\mu}\right),
 \qquad k=0,1,\ldots,
$$
then $S_0$ is exponentially distributed with parameter $\lambda\left(1-{\rm e}^{-\mu}\right)$.  
Moreover, $Z(t)$ is a transient Markov process over the non-negative integers, so   its mean sojourn time 
in state $n$ is finite. Indeed, due to  (\ref{equation:pnt}), 
$$
 \E\{S_n\}= \int_0^{+\infty} p_n(t) \,{\rm d}t
 = \displaystyle{1\over \lambda}\,{\mu^{n}\over n!}\,
\sum_{k=1}^{+\infty} k^n  {\rm e}^{-\mu k}<+\infty.
$$
We conclude this section with an example of application of the iterated Poisson process to queueing theory, 
which also highlights a context in which the results given in the next section are relevant. 
\begin{application}\label{app:1}{ 
Consider a $M/D/1$ queueing system having arrival rate $\lambda$ and  constant unity 
service time. Assume that the service station is inactive during the time interval $[0,t]$, 
so that all waiting customers arriving in this period wait in a queueing room. At time $t$ the bulk 
service of all customers begins, and, thus, ends after a time period whose length is equal 
to the number of customers arrived in $[0,t]$. Assume that during such service time other 
customers join the system with rate $\mu$. Thus, the iterated Poisson process $Z(t)$ describes 
the number of customers that join the system during the service period of those arrived in $[0,t]$. 
In this context, the upward first-crossing time through a constant boundary is the first time that 
the number of customers $Z(t)$ crosses a fixed threshold, such as a waiting-room capacity. 
}\end{application}
%
\section{Stopping time problems for an iterated Poisson process}
Stimulated by the potential applications, in this section we 
study some stopping time problems for the iterated Poisson process. 
\par
The problem of determining the first-passage time distribution for time-changed Poisson processes 
through a constant boundary has been considered recently in \cite{OrTo2013}  with detailed attention 
given to the cases of the fractional Poisson process, the relativistic Poisson process and the gamma 
Poisson process. Moreover, some results on the distributions of stopping times of CPPs in the presence 
of linear boundaries are given in \cite{DiCrMart09}, \cite{StZa03},  \cite{Za91}, and \cite{Za05}.
\par
In this section we confront the first-crossing time problem through various types of boundaries for the iterated 
Poisson process $Z(t)$ investigated in the previous section. 
\par
Let $\beta_k(t)$ be a continuous function such that $\beta_k(t)\geq 0$ for all 
$t\geq 0$, and $\beta_k(0)=k$, with $k\in\mathbb N^+$. The first-crossing time 
of $Z(t)$, defined as in (\ref{eq:defZt}), through the boundary $\beta_k(t)$ will be denoted as
\begin{equation}
 T=\inf\{t>0: Z(t)\geq \beta_k(t)\}. 
 \label{eq:defTk}
\end{equation}
Since $Z(0)=0<\beta_k(0)$, the  first crossing occurs from below. 
\begin{proposition}\label{prop:betadecr}
Let $Z(t)$ be the iterated Poisson process. If $\beta_k(t)$ is nonincreasing in $t$   
then for all $t\geq 0$
\begin{equation}
 \Prob[T>t]=P_{\lfloor \beta_k(t)^- \rfloor}(t),
 \label{eq:PTmt}
\end{equation}
where $\lfloor x^-\rfloor$ is the largest integer smaller than $x$, and $P_n(t)$ has been 
given in (\ref{equation:Pnt}). 
\end{proposition}
\begin{proof}
The proof follows immediately, since the sample paths of $Z(t)$ are nondecreasing, 
and, thus, $\Prob\{T>t\}=\Prob\{Z(t)<\beta_k(t)\}$ for all $t\geq 0$.  
\end{proof}
\par
%
\begin{figure}[t]
\centerline{
\epsfxsize=7.5cm
\epsfbox{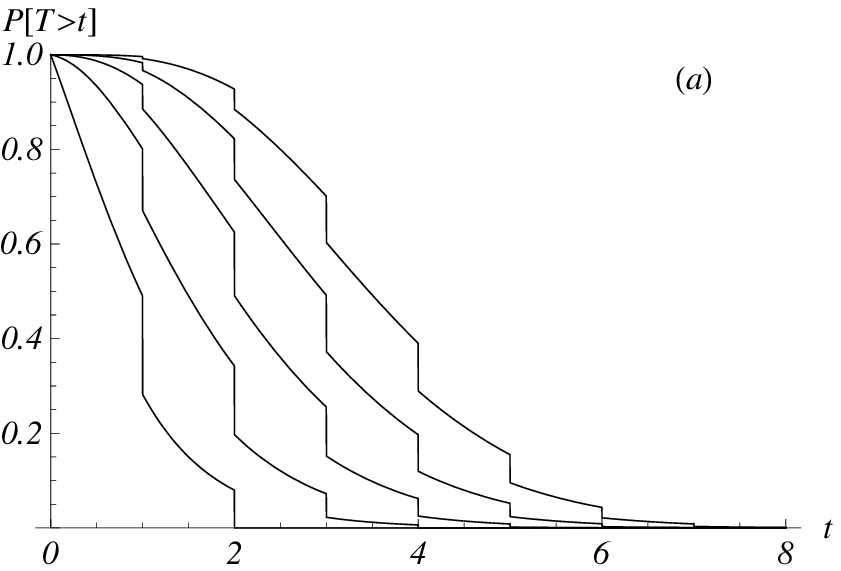}
$\quad$
\epsfxsize=7.5cm
\epsfbox{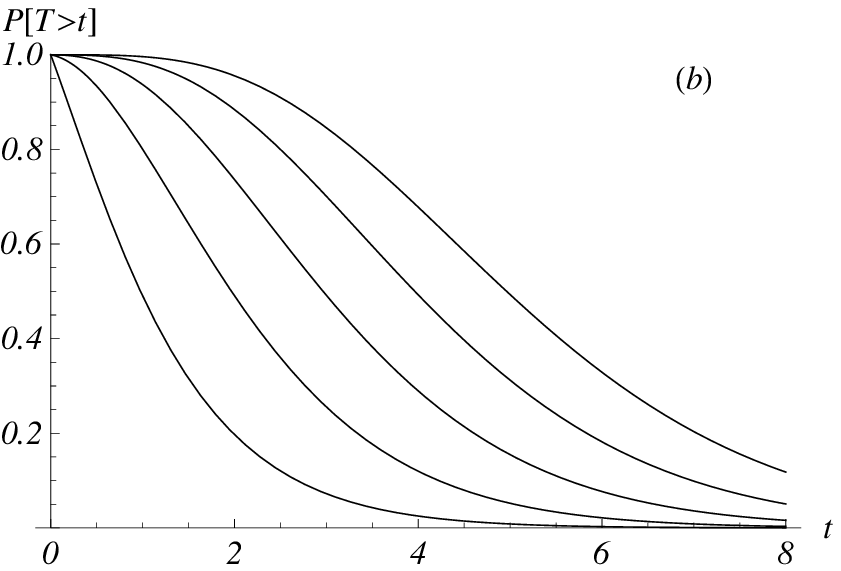}
}
\caption{\small Plots of (\ref{eq:PTmt}) for $k=2$, $4$, $6$, $8$, $10$ (from bottom to top), with  
$\mu=1$, $\lambda=2$, and boundary ({\em a}) $\beta_k(t)=k-t$, and ({\em b}) $\beta_k(t)=k$.}
\label{fig:2casi}
\end{figure}
%
\par
Some plots of the survival function of $T$ are given in Figure \ref{fig:2casi} for a linear decreasing 
boundary and for a constant boundary. 
\subsection{Constant boundary}
We now obtain the probability law of (\ref{eq:defTk}) when the boundary is constant. 
In this case the first-crossing time density of $T$ will be denoted by $\psi(t)$. 
\begin{proposition}\label{prop:fptconst}
Let $Z(t)$ be the iterated Poisson process. If $\beta_k(t)=k$ for all $t\geq 0$ then 
for $k=1$ the first-crossing time density is exponential with parameter 
$\lambda(1-{\rm e}^{-\mu})$, whereas for $k=2,3,\ldots$ 
\begin{align}
 \psi(t)= & \; p_0(t)\,\lambda(1-{\rm e}^{-\mu})
 \bigg[1+  \sum_{i=0}^{k-1} \left(\lambda{\rm e}^{-\mu}t\right)^i C(i;k-1)\bigg] 
 \nonumber \\
 &-p_0(t)\, \lambda {\rm e}^{-\mu}
 \sum_{i=1}^{k-1}i  \left(\lambda{\rm e}^{-\mu}t\right)^{i-1} C(i;k-1)
\label{eq:psibcost}
\end{align}
where 
$$
 C(i;k-1):=\sum_{j=i}^{k-1}S_2(j,i)\,\frac{\mu^j}{j!}.
$$
\end{proposition}
\begin{proof}
From Proposition \ref{prop:betadecr}, which includes the case of constant boundaries, we have 
\begin{equation}
 \Prob\{T>t\}=P_{k-1}(t) \qquad \hbox{for all }t\geq 0.
\label{eq:Survbcost}
\end{equation}
Due to (\ref{eq:Survbcost}) the first-crossing time density through state $k$ can be expressed as 
$\psi(t)=-({\rm d}/{\rm d}t)\Prob\{T>t\}=-({\rm d}/{\rm d}t)P_{k-1}(t)$. 
Hence, differentiating (\ref{equation:Pnt}) with respect to $t$ yields (\ref{eq:psibcost}).
\end{proof}
\par 
We remark that under the assumptions of Proposition \ref{prop:fptconst}, due to (\ref{equation:Pntiniziale}) 
and (\ref{eq:Survbcost}),  $T$ is an honest RV. Let us then evaluate the mean of $T$. 
\begin{corollary}
Let $Z(t)$ be the iterated Poisson process. If $\beta_k(t)=k$ for all $t\geq 0$ then   
\begin{equation}
 \E(T)=\frac{1}{\lambda\left(1-{\rm e}^{-\mu}\right)}\,
 \bigg[1+{\bf 1}_{\{n\geq 1\}} \sum_{k=0}^n \frac{k!}{({\rm e}^{\mu}-1)^k}
 \sum_{j=k}^n S_2(j,k)\, {\mu^{j}\over j!}\bigg],
 \qquad   k\in\mathbb{N}^+.
 \label{eq:MeanTk}
\end{equation}
\end{corollary}
\begin{proof}
Since $\E(T)=\int_0^{+\infty}\Prob[T>t]\,{\rm d}t$, making use of (\ref{equation:Pnt}) 
and (\ref{eq:Survbcost}) after some calculations we obtain (\ref{eq:MeanTk}). 
\end{proof}
\par
Recall that the iterated Poisson process, $Z(t)$, is a Markov process over the non-negative integers. 
We denote by 
$$
 H=\inf\{t>0: Z(t)=k\}, \qquad k\in\mathbb{N}^+,
$$ 
the first-hitting time of state $k$ for $Z(t)$. Let $h(t)$ denote the density of $H$.
\begin{proposition}
Let $Z(t)$ be the iterated Poisson process. 
The density of the first-hitting time $H$ for $t>0$ is given by 
\begin{equation}
 h(t)={{\rm e}^{-\mu} \mu^k\over k!}\,\lambda\,
 \exp\left\{-\lambda\left(1-{\rm e}^{-\mu}\right)t\right\}\,
 B_k'\left(\lambda{\rm e}^{-\mu}t\right), \qquad k=1,2,\ldots,
 \label{eq:densHk}
\end{equation}
where $B'_k(x)$ is defined in (\ref{bellderivata}). 
\end{proposition}
\begin{proof}
For any $t>0$ and $k=1,2,\ldots$ we have 
$$
 h(t)\,{\rm d}t
 =\sum_{j=0}^{k-1}\Prob\{Z(t)=j\}\,\Prob\{Z({\rm d}t)=k-j\}.
$$
Making use of (\ref{equation:pnt}), and by setting $i=j+1$ we obtain 
$$
 h(t) 
 ={{\rm e}^{-\mu} \mu^k\over k!}\,\lambda\,
 \exp\left\{-\lambda\left(1-{\rm e}^{-\mu}\right)t\right\}\,
 \sum_{i=1}^{k} {k\choose i-1}B_{i-1}\left(\lambda{\rm e}^{-\mu}t\right).
$$
The density (\ref{eq:densHk}) then follows by recalling the recurrence equation (\ref{eq:recBn}). 
\end{proof}
\par
Some plots of the first-hitting-time density (\ref{eq:densHk}) are shown in Figure 5.
%
\begin{figure}[t]
\centerline{
\epsfxsize=8cm
\epsfbox{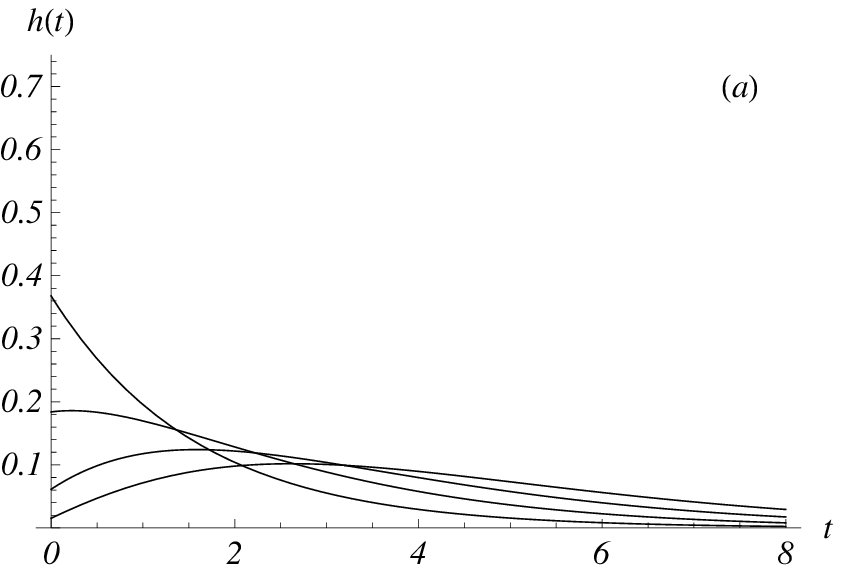}
$\;$
\epsfxsize=8cm
\epsfbox{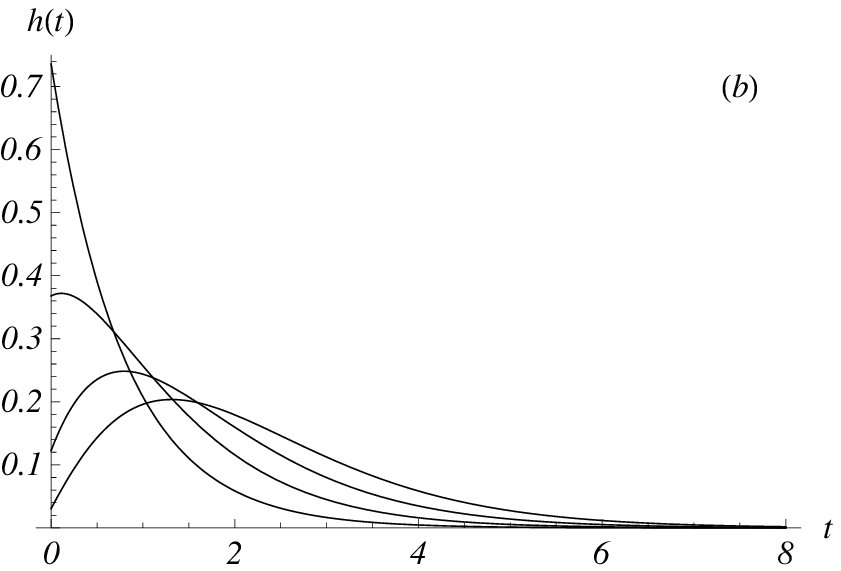}
}
\caption{\small First-hitting time density (\ref{eq:densHk}) with $\mu=1$ and 
(a) $\lambda=1$, and (b) $\lambda=2$, for $k=1,2,3,4$ (from top to bottom near the origin).}
\end{figure}
%
\par
In the following proposition we evaluate the distribution function of the first-hitting time, 
and express it in terms of the lower incomplete gamma function 
$\gamma(a,z)=\int_0^{z}t^{a-1} {\rm e}^{-t}\,{\rm d}t$. 
\begin{proposition}
Let $Z(t)$ be the iterated Poisson process. 
The distribution function of $H$ for $k=1,2,\ldots$ and $t\geq 0$ is given by 
\begin{equation}
 F_{H}(t)={\mu^k\over k!}\Bigg[\exp\left\{-\lambda\big(1-{\rm e}^{-\mu}\big)t\right\}
 B_k\big(\lambda {\rm e}^{-\mu} t \big)
 +\sum_{j=0}^k S_2(k,j)\,{\gamma\left(j+1,\lambda (1-{\rm e}^{-\mu})t\right)
 \over ({\rm e}^{\mu}-1)^j}\Bigg].
 \label{eq:FHk}
\end{equation}
\end{proposition}
\begin{proof}
Making use of Eq.\ (\ref{eq:densHk})  and setting 
$\lambda\,{\rm e}^{\mu}\tau=x$ we have 
$$
 F_{H}(t)=\int_0^t h(\tau)\, {\rm d}\tau
 = {\mu^k\over k!}\int_0^{\lambda {\rm e}^{-\mu} t}
 \exp\left\{-({\rm e}^{\mu}-1)x\right\} B'_k(x)\,{\rm d}x.
$$
Integrating by parts and recalling the Dobi\'nski formula 
(\ref{eq:dobinski}) after some calculations, we finally obtain  (\ref{eq:FHk}). 
\end{proof}
\par
From (\ref{eq:FHk}) it is not hard to see that the first-hitting probability is  
\begin{equation}
 \pi_k:=\Prob(H<\infty)={\mu^k\over k!}\,\sum_{j=1}^k
 S_2(k,j)\,{j!\over ({\rm e}^{\mu}-1)^j}, \qquad k=1,2,\ldots.
 \label{eq:pik}
\end{equation}
It is interesting to note that this probability does not depend on $\lambda$. 
Plots of  $\pi_k=\pi_k(\mu)$ are given in Figure 6 for some choices of $k$. 
Note that the first-hitting probability (\ref{eq:pik}) is in general less than unity.
%
\begin{figure}[t]
\centerline{
\epsfxsize=8cm
\epsfbox{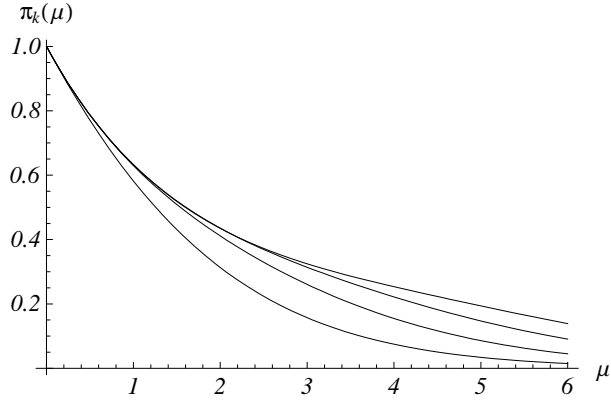}
}
\caption{\small First hitting probability (\ref{eq:pik}) as a function of $\mu$ for 
$k=1,2,3,4$ (from bottom to top).}
\end{figure}
%
\subsection{Linear increasing boundary}
Let us now consider the first-crossing time $T$ of $Z(t)$ through a linear boundary with unity 
slope, $\beta_k(t)=k+t$, where $k$ is a fixed nonnegative integer. For all nonnegative integers 
$j$ we define the avoiding probability 
\begin{equation}
 g(j;t)=\Prob\{Z(t)=j, T>t\}, \qquad t\geq 0. 
 \label{eq:defgk}
\end{equation}
Clearly, we have $g(j;t)=0$ for all $j\geq k+t$. Moreover, for all $t\geq 0$ we have 
$$
 g(j;t)=p_j(t) \qquad \hbox{for }0\leq j\leq k. 
$$
The case $j>k$ will be investigated below. 
\par
We first give a recursive procedure able to evaluate probabilities 
(\ref{eq:defgk}) for integer values of $t$. 
\begin{proposition}\label{proposition:iterg}
Let $Z(t)$ be the iterated Poisson process, and let $\beta_k(t)=k+t$. 
The following steps provide an iterative procedure to evaluate $g(j;n)$, with $n$ integer: 
\par
$\bullet$ for $n=0$: \quad $g(0;0)=1$; 
\par
$\bullet$ for $n=1$: \quad $g(j;1)=p_j(1)$ for $0\leq j\leq k$; 
\par
$\bullet$ for $n=2,3,\ldots$:  
$$
 g(j;n)=\left\{ 
 \begin{array}{ll}
 \displaystyle\sum_{i=0}^j g(i;n-1)\,p_{j-i}(1), & 0\leq j\leq k+n-2 \\[0.5cm]
 \displaystyle\sum_{i=0}^{k+n-2} g(i;n-1)\,p_{k+n-1-i}(1), & j=k+n-1.
 \end{array}
 \right.
$$
\end{proposition}
\begin{proof}
We recall that $p_n(t) = \sum_{j=0}^n p_j(s)\,p_{n-j}(t-s)$, for $0<s<t$. 
Hence, the proof follows by the definitions of the involved quantities and the properties of $Z(t)$. 
\end{proof}
\par
Let us now determine the survival function of $T$ under the assumptions of 
Proposition \ref{proposition:iterg}. 
\begin{proposition}\label{prop:survlincr}
Let $Z(t)$ be the iterated Poisson process, and let $\beta_k(t)=k+t$. 
For any $t> 0$ the survival function of $T$ is given by 
\begin{equation}
 \Prob\{T>t\}=\left\{ 
 \begin{array}{ll}
 \displaystyle\sum_{i=0}^{k+n-1} g(i;n), & t=n \\[0.5cm]
 \displaystyle\sum_{i=0}^{k+n} g(i;t)
 =\sum_{i=0}^{k+n} \sum_{m=0}^i g(m;n)\,p_{i-m}(t-n), & n < t< n+1,
 \end{array}
 \right.
 \label{eq:survblineare}
\end{equation}
where $n$ denotes a nonnegative integer, and in the last sum we set $g(k+n;n)\equiv 0$. 
\end{proposition}
\begin{proof}
From (\ref{eq:defgk}) we have 
$$
 \Prob\{T>t\}=\sum_{j<k+t}g(j;t).
$$
Hence, the proof follows from Proposition \ref{proposition:iterg} and the Markov 
property of $Z(t)$. 
\end{proof}
\par
In conclusion, in Figure 7 we show some plots of the survival function of $T$ obtained by use of 
Proposition \ref{prop:survlincr}.
%
\begin{figure}[t]
\centerline{
\epsfxsize=8cm
\epsfbox{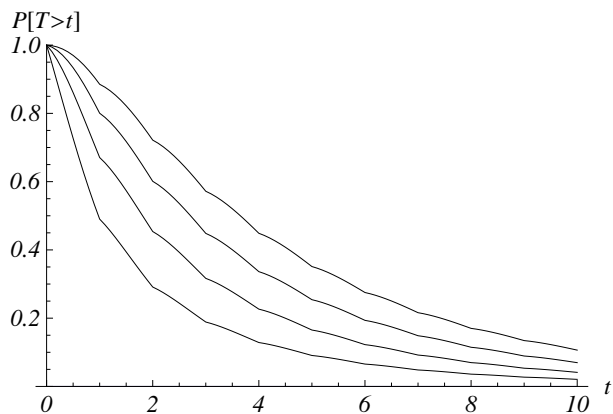}
}
\caption{\small Plot of (\ref{eq:survblineare}) for $\mu=1$ and $\lambda=2$, 
with $k=1,2,3,4$ (from bottom to top), when $\beta_k(t)=k+t$.}
\end{figure}
%
\section*{Acknowledgements} 
The research of A.\ Di Crescenzo and B.\ Martinucci has been performed under partial 
support by Regione Campania (Legge 5) and GNCS-INdAM. S.\ Zacks is grateful to the Ph.D.\ Program 
in Mathematics of Salerno University for his partial support during a visit in July 2010. 

%
\end{document}